\documentclass{amsart}
\usepackage{graphicx}
\newtheorem{thm}{Theorem}[section]
\newtheorem{cor}[thm]{Corollary}
\newtheorem{lem}[thm]{Lemma}

\theoremstyle{definition}
\newtheorem{defn}[thm]{Definition}
\theoremstyle{remark}

\numberwithin{equation}{section}

\begin{document}

\title[Invariant Subspaces of $H^2(\mathbb D^n)$]{Two types of invariant Subspaces in the polydisc}%
\author{Beyaz Ba\c{s}ak KOCA}%
\address{Department of Mathematics, Science Faculty, Istanbul University, 34134, Istanbul, Turkey}%
\email{basakoca@istanbul.edu.tr}%


\begin{abstract}
It is known that the structure of invariant subspaces of the Hardy
space $H^2(\mathbb D^n)$ on the polydisc $\mathbb{D}^n$ is very complicated; hence,
we need good examples help us to understand the structure of
invariant subspaces of $H^2(\mathbb D^n)$. In this paper, we define
two types of invariant subspaces of $H^2(\mathbb D^n)$. Then, we
give a characterization of these types invariant subspaces in view
of the Beurling-Lax-Halmos Theorem. Unitary equivalence is also
studied in this paper.
\end{abstract}
\maketitle
\section{Introduction}
It is well-known that Beurling \cite{Beurling} showed that every
invariant subspaces $M$ of the Hardy space $H^2(\mathbb D)$ on the
unit disc $\mathbb{D}$ is of the form $M=fH^2(\mathbb D)$ for some
inner function $f$ i.e., is generated by a single inner function
(Beurling-type invariant subspaces). However, in the polydisc case,
the structure of the invariant subspaces cannot be characterized in
such a simple form. Although it is clear that the Beurling-type
subspaces are also invariant, determining all invariant subspaces of
$H^2(\mathbb D^n)$ is difficult. In \cite{Jacewicz}, Jacewicz gave
an example of an invariant subspace of $H^2(\mathbb D^2)$ that can
be generated by two functions and can not be generated by a single
function. Later, Rudin also showed in \cite{Rudin} gave an example
of an invariant subspace of $H^2(\mathbb D^2)$ that is not even
finitely generated. Therefore one may naturally ask for a
classification or an explicit description (in some sense) of all
invariant subspaces of $H^2(\mathbb D^n)$. This question was asked
by Rudin in his book \cite[p.78]{Rudin} and it is still open.
Recently, for $n=2$, two types of important invariant subspaces
known as inner-sequence based invariant subspaces and invariant
subspaces generated by two inner functions have been extensively
studied by various authors in different context(see \cite{Yang,
Seto, Seto-Yang, Qin-Yang, Sarkar, Yixin}). In this paper,  inspired
from these studies, we define two new types of invariant subspaces
of $H^2(\mathbb D^n)$ by considering a larger class of functions
than inner functions. Then, we deal with the structure of these
invariant subspaces in view of the Beurling-Lax-Halmos theorem. Our
method is the same with the work done in \cite{Qin-Yang}. However,
our examples of invariant subspaces and results related to them
improve
and generalize results proved for $n=2$ in \cite{Qin-Yang}.\\

Before starting, we will give preliminary definitions and few
important results that we will use in this study.\\

Let $n$ be a positive integer. The open unit disc in $\mathbb{C}$
is denoted by $\mathbb D$; its boundary is the circle $T$. The polydisc
$\mathbb D^n$ and its distinguished boundary, the torus, $T^n$ are the subsets of $\mathbb{C}^n$
$(n>1)$ which are cartesian products of $n$ copies of $\mathbb D$
and $T$, respectively.

The Hardy space on the polydisc $H^2(\mathbb D^n)$ is defined as
\[H^2(\mathbb{D}^n)=\{f(z): z\in\mathbb D^n, f(z)=\sum_\alpha c(\alpha)z^\alpha,\; \alpha\in \mathbb{Z}_+^n,\; \sum_\alpha |c(\alpha)|^2<\infty \}\]
In fact, $||f||_2=\{\sum_\alpha |c(\alpha)|^2\}^{1/2}$. For $f(z)\in H^2(\mathbb D^n)$, the radial limit
\[f^*(z)=\lim_{r\rightarrow 1}f(rz)\]
exists at almost every $z\in T^n$.

$H^\infty(\mathbb D^n)$ is the space of all
bounded analytic functions in $\mathbb D^n$;
$||f||_\infty=\sup_{z\in\mathbb D^n}|f(z)|$.
An inner function in $\mathbb D^n$ is a function $g\in
H^\infty(\mathbb D^n)$ with $|g^*|=1$ a.e. on $T^n$.

Recall a subspace $M$ of $H^2(\mathbb D^n)$ is called ``invariant'' if (a) $M$ is closed linear subspace of $H^2(\mathbb D^n)$ and (b) $f\in M$ implies $z_if\in M$ for $i=1,\ldots,n$; i.e., multiplication by the variables $z_1,z_2,\ldots,z_n$ maps $M$ into $M$. The smallest invariant subspace of $H^2(\mathbb D^n)$ which contains a given $f$ is denoted by $M_f$ and $M_f$ is called the
subspace generated by $f$ if $M_f=fH^2(\mathbb D^n)$.  For further information for Hardy space on the polydisc, see \cite{Rudin}.

In \cite{Kocainvariant}, the author and Sad\i k completely
characterized the singly-generated invariant subspaces of
$H^2(\mathbb{D}^n)$ as follows:
\begin{thm}{\cite[Theorem 2.1]{Kocainvariant}}\label{koca}
Let $f\in H^\infty(U^n)$. The subspace $f H^2(U^n)$ of $H^2(U^n)$ is invariant if and only if $f$ is a generalized inner function.
\end{thm}
Here a generalized inner function means that $f\in H^\infty(U^n)$
with $f^{-1}\in L^\infty(T^n)$. Here it is clear that
$f^{-1}=1/f^*$. The authors also constructed a singly generated
invariant subspace that can not be Beurling type \cite[Theorem 2.3]{Kocainvariant}.\\

We recall the class of analytic vector valued functions. Let $K$ be
an infinite dimensional separable Hilbert space. Then
\[H^2(K)=\{f(z): z\in\mathbb D, f(z)=\sum_{n=0}^\infty x_nz^n,\; x_n\in K,\; \sum_{n=0}^\infty||x_n||_K^2<\infty \},\]
where $||\cdot||_K$ denotes the norm of the space $K$. Clearly,
$H^2(K)$ is a Hilbert space under the inner product
\[(f\cdot g)= \sum_{n=0}^\infty(x_ny_n), \mbox{ where } f(z)=\sum_{n=0}^\infty x_nz^n \mbox{ and } g(z)=\sum_{n=0}^\infty y_nz^n.\]
Also, $||f||^2_{H^2(K)}=\sum_{n=0}^\infty||x_n||_K^2$.\\

The set of all bounded operator-valued analytic functions
on $\mathbb D$ with values in the algebra $B(K)$ of bounded linear
operators on the space $K$ is defined by
\[H^\infty(B(K))=\{W: W(z)=\sum_{n=0}^\infty A_nz^n,\; z\in\mathbb D,\; A_n\in B(K), \sup_{z\in\mathbb D}||W(z)||_{B(K)}<\infty\}.\]
Also, $||W||_\infty=\sup_{z\in\mathbb D}||W(z)||_{B(K)}$.\\

It is obvious that every element $W\in H^\infty(B(K))$ gives rise to
a bounded linear operator to $H^2(K)$, i.e, to an element $W$, we
correspond an operator $\hat{W}$ on $H^2(K)$ that is defined by the
formula
\[(\hat W\varphi)(z)=W(z)\varphi(z),\; z\in\mathbb D.\]
For more detail on the space of all vector-valued analytic
functions, see \cite{riesz}.\\

A function $W\in H^\infty(B(K))$ is called  operator inner if the pointwise a.e. boundary values are isometries:
\[(W(\xi))^*W(\xi)=I_K \mbox{ for almost all }\xi\in\mathbb{T}.\]

\begin{thm}{\cite[Beurling-Lax-Halmos Theorem]{riesz}}\label{laks}
A subspace $M$ of $H^2(K)$ is invariant under the multiplication
operators by the independent variable if and only if $M$ is of the form $M=\hat
\Theta H^2(K)$ for some operator inner function $\Theta$.
\end{thm}
The following property is well-known:
\begin{lem}\label{lemma1}
For $n>1$, $H^2(\mathbb D^n)=H^2(H^2(\mathbb D^{n-1}))$
\end{lem}
Combining with Lemma \eqref{lemma1} and Theorem \eqref{laks}, we obtain
the full description of invariant subspace of $H^2(\mathbb D^n)$ under
the multiplication operator by the variable $z_1$ as
follows:
\begin{cor}\label{lemmalaks}
A subspace $M$ of $H^2(\mathbb D^n)=H^2(H^2(\mathbb D^{n-1}))$ is
invariant under the multiplication operator by the
variable $z_1$ if and only if $M=\hat \Theta
H^2(H^2(\mathbb D^{n-1}))$ for some operator inner function $\Theta\in H^\infty(B(H^2(\mathbb D^{n-1})))$.
\end{cor}

\section{Sequence-Based Invariant Subspaces}
\begin{defn}\label{defn}
An invariant subspace $M$ of $H^2(\mathbb D^n)$ is called
sequence-based if it is of the form
\begin{equation}\label{ex1}
M=\bigoplus_{l\geq 0}
f_l(z_1,\ldots,z_{n-1})H^2(\mathbb{D}^{n-1})z_n^l,
\end{equation}
where the sequence $\{f_l\}_{l\geq 0}$ consists of functions having
the following properties:
\begin{itemize}
\item[(I)] $f_l\in H^\infty(\mathbb D^{n-1})$ with $f_l^{-1}\in L^\infty(\mathbb T^{n-1})$ for any $l$,
\item[(II)] $f_l$ is divisible by $f_{l+1}$ for any $l$, i.e.,
every $(f_l/f_{l+1})$ satisfies the condition (I).
\end{itemize}
\end{defn}
It is clear that $M$ is invariant under the multiplication
by the variables $z_1,\ldots,z_{n-1}$. Moreover, note
that the condition (II) is equivalent to $f_lH^2(\mathbb D^{n-1})$
is contained in $f_{l+1}H^2(\mathbb D^{n-1})$. From this, we have
\[z_nf_lH^2(\mathbb D^{n-1})z_n^l=f_lH^2(\mathbb D^{n-1})z_n^{l+1}\subset f_{l+1}H^2(\mathbb D^{n-1})z_n^{l+1}.\]
This shows that $M$ is also invariant under the multiplication by $z_n$.\\

Remark that inner functions is properly contained in the class of
all functions $f\in H^\infty(\mathbb{D}^n)$ with $f^{-1}\in
L^\infty(T^n)$ for any $n>1$. 
In the case of $n=2$, the inner sequence based invariant subspaces are studied. The characterization of
this type of invariant subspaces is studied by Qin and Yang in
\cite[Theorem 2.1]{Qin-Yang}. We now give the following
characterization of
sequence- based invariant subspaces as same manner of their proof.\\

%

Since the subspace $M=\bigoplus_{l\geq 0}
f_l(z_1,\ldots,z_{n-1})H^2(\mathbb D^{n-1})z_n^l$ is invariant,
there exists an operator inner function $\Theta(z_n)\in
H^\infty(B(H^2(\mathbb D^{n-1})))$ by the Beurling-Lax-Halmos
theorem. Assume its power series representation
\[\Theta(z_n)=\sum_{l\geq} P_lz_n^l,\]
where $z_n\in \mathbb D$ and $P_l$ are operators on $H^2(\mathbb
D^{n-1})$.

\begin{thm}
Let $\Theta(z_n)=\sum_{l\geq0} P_lz_n^l$ be the operator inner
function for an invariant subspace $M$ of $H^2(\mathbb D^n)$. Then
$M$ is sequence-based if and only if $P_l$, $l\geq0$, are orthogonal
projections on $H^2(\mathbb D^{n-1})$ with perpendicular ranges such
that $\bigoplus_{l=0}^kP_lH^2(\mathbb D^{n-1})$, $k\geq0$, is
generated by a single function.
\end{thm}

\begin{proof}
Suppose that $M=\bigoplus_{l\geq
0}f_l(z_1,\ldots,z_{n-1})H^2(\mathbb D^{n-1})z_n^l$. It is easy
compute that
\[M\ominus z_nM=f_0H^2(\mathbb D^{n-1})\oplus\bigoplus_{l\geq 1} z_n^l(f_lH^2(\mathbb D^{n-1})\ominus f_{l-1}H^2(\mathbb D^{n-1})).\]
For simplicity we let $N_l=f_lH^2(\mathbb D^{n-1})\ominus
f_{l-1}H^2(\mathbb D^{n-1})$, $l\geq 1$. Let $P_l$ be the orthogonal
projection from $H^2(\mathbb D^{n-1})$ onto $N_l$, $l\geq 1$, $P_0$
be the orthogonal projection from $H^2(\mathbb D^{n-1})$ onto
$f_0H^2(\mathbb D^{n-1})$, and set
$\Theta(z_n)=\sum_{l\geq0} P_lz_n^l$.\\

For any $g(z_1,\ldots,z_{n-1})\in H^2(\mathbb D^{n-1})$,
$\Theta(z_n)g=\sum_{l\geq0}z_n^lP_lg$ and
\[||\Theta(z_n)g||^2=||\sum_{l\geq0}z_n^lP_lg||^2=\sum_{l\geq0}||z_n^lP_lg||^2=||g||^2\]
This equality shows that $\Theta(z_n)$ is an operator inner
function.
\begin{equation*}
\begin{split}
M\ominus z_nM &=f_0H^2(\mathbb D^{n-1})\oplus\bigoplus_{l\geq 1}z_n^lN_l\\
&=P_0H^2(\mathbb D^{n-1})\oplus\bigoplus_{l\geq 1}z_n^lP_lH^2(\mathbb D^{n-1})\\
&=\Theta(z_n)H^2(\mathbb D^{n-1}).
\end{split}
\end{equation*}
By this, we have
\begin{equation*}
\begin{split}
M &=\bigoplus_{l\geq 0}z_n^l(M\ominus z_nM)=\bigoplus_{l\geq 0}z_n^l\Theta(z_n)H^2(\mathbb D^{n-1})\\
&=\Theta(z_n)\bigoplus_{l\geq 0}z_n^lH^2(\mathbb
D^{n-1})=\Theta(z_n)H^2(\mathbb D^{n}).
\end{split}
\end{equation*}
Conversely, suppose $P_l$, $l\geq0$ are orthogonal projections on
$H^2(\mathbb D^{n-1})$ with perpendicular ranges such that
$\bigoplus_{l=0}^kP_lH^2(\mathbb D^{n-1})$, $k\geq0$, is generated
by a single function and $\Theta(z_n)=\sum_{l\geq0}z_n^lP_l$. Then
\begin{equation}\label{eq1}
M=\Theta(z_n)H^2(\mathbb
D^n)=\sum_{l\geq0}z_n^lP_l\left(\bigoplus_{t\geq 0} z_n^tH^2(\mathbb
D^{n-1})\right)=\sum_{k=0}^\infty z_n^k\left(\bigoplus_{l=0}^k
P_lH^2(\mathbb D^{n-1})\right)
\end{equation}
Since $z_iM\subseteq M$, $i=1,\ldots,n-1$, each closed subspace
$\bigoplus_{l=0}^kP_lH^2(\mathbb D^{n-1})$, $k\geq0$ is an invariant
subspace and by assumption $\bigoplus_{l=0}^kP_lH^2(\mathbb
D^{n-1})$, $k\geq0$ is an invariant subspace generated by a single
function. By Theorem \eqref{koca} there exists a function $f_k\in
H^\infty(\mathbb D^{n-1})$ with $f_k^{-1}\in L^\infty(\mathbb
T^{n-1})$ such that
\[\bigoplus_{l=0}^kP_lH^2(\mathbb D^{n-1})=f_k H^2(\mathbb D^{n-1}).\]
Clearly, $f_k$ is divisible by $f_{k+1}$ for any $k$. Hence
$\{f_k\}$ satisfies the conditions (I) and (II) in Definition
\eqref{defn}, and by \eqref{eq1} we have
\[M=\bigoplus_{k\geq 0} f_kH^2(\mathbb D^{n-1})z_n^k,\]
that is $M$ is sequence-based invariant subspace.
\end{proof}

\section{Invariant Subspaces Generated by Two Functions}
In this section we deal with the invariant subspace $M$ of the form
\begin{equation}\label{ex2}
M= f_1(z_1)H^2(\mathbb D^n)+f_2(z_2,\ldots,z_n)H^2(\mathbb D^n),
\end{equation}
where $f_1(z_1)\in H^\infty(\mathbb D)$ with $f_1^{-1}\in
L^\infty(\mathbb T)$ and $f_2(z_2,\ldots,z_n)\in H^\infty(\mathbb
D^{n-1})$ with $f_2^{-1}\in L^\infty(\mathbb T^{n-1})$. 
To see invariance of $M$ it is enough to show that $M$ is closed as in
\cite[Lemma 2.4]{Izuchi}. In fact, since
\[H^2(\mathbb{D}^n)\ominus
f_2H^2(\mathbb{D}^n)=\bigoplus_{l\geq0}z_l(H^2(\mathbb{D}^{n-1})\ominus
f_2H^2(\mathbb{D}^{n-1})),\] $H^2(\mathbb{D}^n)\ominus
f_2H^2(\mathbb{D}^n)$ is invariant under the multiplication operator
by $z_1$. Then $f_1(z_1)(H^2(\mathbb{D}^n)\ominus
f_2H^2(\mathbb{D}^n))\perp f_2H^2(\mathbb{D}^n)$ and
\begin{equation*}
\begin{split}
M &= f_1H^2(\mathbb D^n)+f_2H^2(\mathbb D^n)\\
&=f_1(H^2(\mathbb{D}^n)\ominus f_2H^2(\mathbb{D}^n))\oplus
f_2H^2(\mathbb{D}^n) +f_2H^2(\mathbb{D}^n)\\
&= f_1(H^2(\mathbb{D}^n)\ominus f_2H^2(\mathbb{D}^n))\oplus
f_2H^2(\mathbb{D}^n).
\end{split}
\end{equation*}
Hence $M$ is closed. \\

For $n=2$, the case of $f_1$ and $f_2$ are inner was studied in \cite{Izuchi,Yang, Qin-Yang}. Following their method
in \cite{Qin-Yang}, we characterize $M$ of the form \eqref{ex2} in
terms of $\Theta(z_1)$ corresponding to $M$.\\

Before starting, we need the following lemma:
\begin{lem}\label{lemma}
If $M$ is an invariant subspace of the form \eqref{ex2}, then
\[M\ominus z_1M=f_1(H^2(\mathbb D^{n-1})\ominus f_2H^2(\mathbb D^{n-1}))\oplus f_2H^2(\mathbb D^{n-1}).\]
\end{lem}
\begin{proof}
We prove the lemma as in \cite[Lemma 3.2]{Yang}. Given a pair of
functions $f_1(z_1)$ and $f_2(z_2,\ldots,z_n)$ corresponding to $M$,
we can decompose $H^2(\mathbb D^n)$ as

\[\left((H^2(\mathbb D)\ominus f_1H^2(\mathbb D))\oplus f_1H^2(\mathbb D)\right)\otimes
\left((H^2(\mathbb D^{n-1})\ominus f_2H^2(\mathbb D^{n-1}))\oplus
f_2H^2(\mathbb D^{n-1})\right)\]\[=(H^2(\mathbb D)\ominus
f_1H^2(\mathbb D))\otimes(H^2(\mathbb D^{n-1})\ominus f_2H^2(\mathbb
D^{n-1}))\oplus (H^2(\mathbb D)\ominus f_1H^2(\mathbb D))\otimes
f_2H^2(\mathbb D^{n-1})\]\[ \oplus f_1H^2(\mathbb
D)\otimes(H^2(\mathbb D^{n-1})\ominus f_2H^2(\mathbb D^{n-1}))\oplus
f_1H^2(\mathbb D)\otimes f_2H^2(\mathbb D^{n-1}).\]

By this equality we have
\begin{equation*}
(H^2(\mathbb D)\ominus f_1H^2(\mathbb D))\otimes(H^2(\mathbb
D^{n-1})\ominus f_2H^2(\mathbb D^{n-1}))=H^2(\mathbb D^n)\ominus
(f_1H^2(\mathbb D)+f_2H^2(\mathbb D^{n-1})).
\end{equation*}
Relative to the decomposition
\begin{equation*}
\begin{split}
M &=f_1H^2(\mathbb D)+f_2H^2(\mathbb D^{n-1})\\
&= (H^2(\mathbb D)\ominus f_1H^2(\mathbb D))\otimes f_2H^2(\mathbb
D^{n-1})\oplus f_1H^2(\mathbb D)\otimes (H^2(\mathbb D^{n-1})\ominus
f_2H^2(\mathbb D^{n-1}))\\
& \oplus f_1H^2(\mathbb D)\otimes f_2H^2(\mathbb D^{n-1}).
\end{split}
\end{equation*}
Based on this equality, we can write $M$ as
\[M=f_1H^2(\mathbb D)\otimes (H^2(\mathbb D^{n-1})\ominus f_2H^2(\mathbb
D^{n-1}))\oplus H^2(\mathbb D)\otimes f_2H^2(\mathbb D^{n-1})\] and
the lemma follows easily.
\end{proof}

\begin{thm}\label{thm2}
Let $\Theta(z_1)$ be the operator inner function for an invariant
subspace $M$. Then $M$ is of the form \eqref{ex2} if and only if
$\Theta(z_1)=f_1(z_1)P_0+P_1$, where $P_1$ is a projection from
$H^2(\mathbb D^{n-1})$ to an invariant subspace generated by a
single function and $P_0$ is a complemented projection of $P_1$,
i.e., $P_0P_1=0$, $P_0+P_1=I$.
\end{thm}
\begin{proof}
Suppose that $M$ is of the form \eqref{ex2}. By Lemma \eqref{lemma}
we have
\[M\ominus z_1M=f_1(H^2(\mathbb D^{n-1})\ominus f_2H^2(\mathbb D^{n-1}))\oplus f_2H^2(\mathbb D^{n-1}).\]
Set $\Theta(z_1)=f_1P_0+P_1$, where $P_0:H^2(\mathbb
D^{n-1})\rightarrow H^2(\mathbb D^{n-1})\ominus f_2H^2(\mathbb
D^{n-1})$ is the orthogonal projection and $P_1=I-P_0$. Then, for
every $h\in H^2(\mathbb D^{n-1})$, by the Pythagorean theorem,
\begin{equation*}
||\Theta(z_1)h||^2=|f_1(z_1)|^2||P_0h||^2+||P_1h||^2=||P_0h||^2+||P_1h||^2=||h||^2.
\end{equation*}
This shows that $\Theta$ is an operator inner function. Further, we
have \[\Theta(z_1)H^2(\mathbb D^{n-1})=f_1P_0H^2(\mathbb
D^{n-1})+P_1H^2(\mathbb D^{n-1})=M\ominus z_1M,\] and hence
\[M=\bigoplus_{n=0}^\infty z_1^n(M\ominus z_1M)=\Theta(z_1)H^2(\mathbb D^{n}).\]
Conversely, suppose $\Theta(z_1)=f_1(z_1)P_0+P_1$, where $f_1$ is
inner, $P_1$ is a projection from $H^2(\mathbb D^{n-1})$ to an
invariant subspace generated by a single function and $P_0$ is a
complemented projection of $P_1$. Then
\begin{equation}\label{eq2}
M=\Theta(z_1)H^2(\mathbb D^{n})=f_1(H^2(\mathbb D)\otimes
P_0H^2(\mathbb D^{n-1}))\oplus (H^2(\mathbb D)\otimes P_1H^2(\mathbb
D^{n-1}))
\end{equation}
First, we show that for all $i=2,\ldots, n$, the multiplication
operators by $z_i$, $T_{z_i}$ and $P_1$ commute on $M$. Denote
$M_0=H^2(\mathbb D)\otimes P_0H^2(\mathbb D^{n-1})$ and
$M_1=H^2(\mathbb D)\otimes P_1H^2(\mathbb D^{n-1})$, and let
$P_{M_0}$ and $P_{M_1}$ stand for the projections from $H^2(\mathbb
D^n)$ to $M_0$ and $M_1$, respectively. Then, with respect to the
decomposition \eqref{eq2} we rewrite $T_{z_i}$, $i=2,\ldots,n$ on
$M$ as
\begin{equation*}
T_{z_i}=\left(%
\begin{array}{cc}
  P_{M_0}T_{z_i}P_{M_0} &  P_{M_0}T_{z_i}P_{M_1} \\
  P_{M_1}T_{z_i}P_{M_0} &  P_{M_1}T_{z_i}P_{M_1} \\
\end{array}%
\right),  i=2,\ldots, n
\end{equation*}
Since $M$ is invariant under $T_{z_i}$, $i=2,\ldots,n$, we have
\begin{equation*}
\left(%
\begin{array}{cc}
  P_{M_0}T_{z_i}P_{M_0} &  P_{M_0}T_{z_i}P_{M_1} \\
  P_{M_1}T_{z_i}P_{M_0} &  P_{M_1}T_{z_i}P_{M_1} \\
\end{array}%
\right)\left(%
\begin{array}{c}
  f_1M_0 \\
  M_1 \\
\end{array}%
\right)\subset \left(%
\begin{array}{c}
  f_1M_0 \\
  M_1 \\
\end{array}%
\right),
\end{equation*}
i.e.,
\begin{equation}\label{eq3}
\left(%
\begin{array}{c}
  f_1P_{M_0}T_{z_i}M_0 +  P_{M_0}T_{z_i}M_1 \\
  f_1P_{M_1}T_{z_i}M_0 +  P_{M_1}T_{z_i}M_1 \\
\end{array}%
\right)\subset \left(%
\begin{array}{c}
  f_1M_0 \\
  M_1 \\
\end{array}%
\right).
\end{equation}
Consider the first line in \eqref{eq3}. It is clear that
$f_1P_{M_0}T_{z_i}M_0\subset f_1M_0$, and hence
$P_{M_0}T_{z_i}M_1\subset f_1M_0$. It is easy computed that
$P_{M_0}T_{z_i}M_1=H^2(\mathbb D)\otimes P_0z_iP_1H^2(\mathbb
D^{n-1})$ and $f_1M_0=f_1H^2(\mathbb D)\otimes P_0H^2(\mathbb
D^{n-1})$. Therefore, since $f_1$ is non-trivial, the first
inclusion in \eqref{eq3} holds only if $P_0z_iP_1H^2(\mathbb
D^{n-1})=\{0\}$, $i=2,\ldots,n$. This implies $z_iP_1H^2(\mathbb
D^{n-1})\subset P_1H^2(\mathbb D^{n-1})$, $i=2,\ldots,n$, i.e.,
$P_1H^2(\mathbb D^{n-1})$ is invariant subspace of $H^2(\mathbb
D^{n-1})$ and by assumption it is generated by a single function
$f_2(z_2,\ldots,z_n)\in H^\infty(\mathbb D^{n-1})$ ,i.e.,
$P_1H^2(\mathbb D^{n-1})=f_2H^2(\mathbb D^{n-1})$. It follows from
Theorem \eqref{koca} that $f_2^{-1}\in L^\infty(\mathbb T^{n-1})$.
Finally, we have
\begin{equation*}
M=f_1(H^2(\mathbb D)\otimes P_0H^2(\mathbb D^{n-1}))\oplus
(H^2(\mathbb D)\otimes f_2H^2(\mathbb D^{n-1}))=f_1H^2(\mathbb
D^n)+f_2 H^2(\mathbb D^n).
\end{equation*}
\end{proof}

\section{Unitary Equivalence}
Two invariant subspaces $M_1$ and $M_2$ of $H^2(\mathbb D^n)$ are
said to be unitarily equivalent if there is a unitary operator
$U:M_1\rightarrow M_2$ such that $U(\varphi f)=\varphi(Uf)$ for
$\varphi\in H^\infty(\mathbb D^n)$ and $f\in M_1$. Agrawal, Clark
and Douglas \cite{Agrawal} study the question of unitary equivalence
of invariant subspaces of $H^2(\mathbb{D}^n)$. Specifically, unitary
equivalence of inner-based invariant subspaces and invariant
subspaces generated by two inner functions of $H^2(\mathbb D^2)$ are
studied by Seto \cite{Seto} and Yang \cite{Yang}, respectively. In
this section we determine unitary equivalence of sequence-based
invariant subspaces and invariant subspaces generated by two
functions of $H^2(\mathbb D^n)$, separately.

\begin{thm}
Let $M_1$ and $M_2$ denoted sequence-based invariant subspace of
$H^2(\mathbb D^n)$ corresponding to sequences
$\{f_l(z_1,\ldots,z_{n-1})\}_{l\geq 0}$ and
$\{g_l(z_1,\ldots,z_{n-1})\}_{l\geq 0}$, respectively. Then $M_1$
and $M_2$ unitarily equivalent if and only if there exists a
unimodular function $h(z_1,\ldots,z_{n-1})$ depending only variables
$z_1,\ldots,z_{n-1}$ such that $M_2=hM_1$.
\end{thm}

\begin{proof}
If $M_1$ and $M_2$ are unitarily equivalent, there exists a
unimodular function $h(z_1,\ldots,z_n)$ such that $M_2=hM_1$ by
Lemma 1 in \cite{Agrawal}. Since $\overline{h}g_0$ and $hf_0$ are in
$H^2(\mathbb D^n)$, $h$ is $z_n$-analytic and conjugate
$z_n$-analytic. Hence $h$ depends only variables
$z_1,\ldots,z_{n-1}$. The converse is trivial.

%
\end{proof}

\begin{thm}
Let $f_1(z_1), g_1(z_1)$ be inner functions and
$f_2(z_2,\ldots,z_n), g_2(z_2,\ldots,z_n)$ be functions in
$H^\infty(\mathbb D^{n-1})$ with $f_2^{-1}, g_2^{-1} \in
L^\infty(\mathbb T^{n-1})$ and
\[M_1= f_1H^2(\mathbb D^n)+f_2H^2(\mathbb D^n),\;\; M_2= g_1H^2(\mathbb D^n)+g_2H^2(\mathbb D^n).\]
Then $M_1$ is unitarily equivalent to $M_2$ only if $M_1=M_2$.
\end{thm}

\begin{proof}
By Lemma 2.1 in \cite{Yang}, there is an inner function $\phi$ such
that $M_2=\phi M_1$. Since it implies $M_1=\bar\phi M_2$, $\bar\phi$
is also inner. Then $\phi$ is constant. This proves the theorem.
\end{proof}

\bibliographystyle{amsplain}

\end{document}